\documentclass{amsart}
\usepackage{amsmath,amsthm,amscd}

\usepackage{amssymb}

\newtheorem{thm}{Theorem}

\newtheorem{lem}[thm]{Lemma}

\newtheorem{prop}[thm]{Proposition}

\newtheorem{conj}[thm]{Conjecture}
   
\theoremstyle{definition}
\newtheorem{defn}[thm]{Definition}

\newtheorem{say}[thm]{}
\newtheorem{exmp}[thm]{Example}

\newtheorem{ques}[thm]{Question}    

\newtheorem{rem}[thm]{Remark}          

\newtheorem*{ack}{Acknowledgments}      

\newtheorem{defn-thm}[thm]{Definition--Theorem}  
\newtheorem{defn-lem}[thm]{Definition--Lemma}  

\theoremstyle{remark}

\let \cedilla =\c
\renewcommand{\c}[0]{{\mathbb C}}

\newcommand{\z}[0]{{\mathbb Z}}

\renewcommand{\r}[0]{{\mathbb R}}

\newcommand{\p}[0]{{\mathbb P}}

\newcommand{\bd}[0]{{\mathbb D}}

\newcommand{\qtq}[1]{\quad\mbox{#1}\quad}

\newcommand{\im}[0]{\operatorname{im}}

\newcommand{\chow}[0]{\operatorname{Chow}}

\newcommand{\onto}[0]{\twoheadrightarrow}

\newcommand{\univ}[0]{\operatorname{Univ}}

\newcommand{\alb}[0]{\operatorname{alb}} 
\newcommand{\Alb}[0]{\operatorname{Alb}}

\def\into{\DOTSB\lhook\joinrel\to}

\begin{document}
\bibliographystyle{amsalpha}


\title[Algebraic varieties with semialgebraic universal cover]{Algebraic 
varieties with\\ semialgebraic universal cover}
\author{J\'anos Koll\'ar}
\author{John Pardon}

\maketitle

The study of algebraic varieties using their universal cover
goes back to Euler (for elliptic curves), Abel (for Abelian varieties)
and Poincar\'e (for curves of genus $\geq 2$). 
A  general  approach was initiated by
 Shafarevich \cite[Sec.IX.4]{shaf};
see \cite{shaf-book, MR1492537, BK98, ekpr} and the references there
for later results and surveys.

In the classical examples,  the universal cover is rather simple
($\c^n$ or a bounded symmetric domain), which makes it possible 
to get  detailed information about a variety using its universal cover.
This leads  to our basic question:
$$
\mbox{\it Which smooth projective varieties have ``simple'' universal cover?}
$$
There are at least three  ways to define what ``simple'' should mean,
depending on which properties of the  universal cover $\tilde X$
we focus on.
\begin{enumerate}
\item[$\bullet$] (Topology)  $\tilde X$ is 
homotopic to a finite CW complex.
\item[$\bullet$] (Complex analysis) $\tilde X$ is a bounded domain in a
 Stein manifold. 
\item[$\bullet$] (Algebraic geometry) $\tilde X$ is an algebraic variety.
\end{enumerate}

The algebro-geometric variant was investigated in \cite{chk} whose
 main result roughly 
 says that all such examples are obtained from 
Abelian varieties.

\begin{thm} \cite[Thm.1]{chk} \label{theoremquasiprojective}
Let $X$ be a normal, projective variety over $\c$ 
with  universal cover $\tilde X$.
The following are equivalent.
\begin{enumerate}
\item  $\tilde X$ is  biholomorphic to a  
quasi-projective variety.
\item  $\tilde X$ is biholomorphic to a 
product $\c^m \times F$ where $m\geq 0$
and $F$ is a projective, simply connected variety.
\end{enumerate}
\end{thm}

(The proof in \cite{chk} assumes the validity of the
Abundance Conjecture \cite{reid}, so the theorem is unconditionally established 
only if $\dim X\leq 3$.)

Our aim is to extend these results to a larger class of examples, that includes 
 the compact quotients of  bounded, symmetric domains.
Let $\bd\subset \c^n$ be a  bounded, symmetric domain in any of its
usual representations. Then $\bd$ is not quasi-projective,
but it is {\it semialgebraic}. That is, if we set $z_j=x_j+iy_j$, then 
$\bd $ can be defined by 
polynomial inequalities of the form $f(x_1,\dots, x_n,y_1,\dots, y_n)>0$;
see (\ref{bunded.defn}) and (\ref{semialg.defns}) for details.
This leads to the following.

\begin{conj} \label{semi.main.conj}
Let $X$ be a normal, projective variety over $\c$
with  universal cover $\tilde X$.
The following are equivalent.
\begin{enumerate}
\item $\tilde X$ is  biholomorphic to a  
semialgebraic open subset of a projective variety.
\item  $\tilde X$ is biholomorphic to a 
product $\bd\times \c^m \times F$ where $\bd$ is a bounded
symmetric domain, $m\geq 0$
and $F$ is a normal, projective, simply connected variety.
\end{enumerate}
\end{conj}

The first part of the proof of \cite{chk}
shows that if (\ref{theoremquasiprojective}.1) holds then
$\pi_1(X)$ has a finite index abelian subgroup.
The second part then shows that  Theorem \ref{theoremquasiprojective}
holds if  $\pi_1(X)$ is abelian.

Our main result is a non-abelian version of this second step.
In effect we prove Conjecture  \ref{semi.main.conj}
in case $\pi_1(X)$ acts properly discontinuously and freely
with compact quotient  on an
irreducible, bounded, symmetric domain $\bd$.
A three step approach to the proof of  Conjecture  \ref{semi.main.conj}
is outlined in Section \ref{sec.conjectures}. Our results establish Step 3
and most of Step 2. However, we say nothing about
Step 1 which seems to be the hardest.
On the other hand, we also describe certain intermediate covers
in a precise way.

\begin{thm}  \label{main.thm.quot-version}
Let $\Gamma$ be  a group acting properly discontinuously and freely
with compact quotient  on an
irreducible, bounded, symmetric domain $\bd$ of dimension $\geq 2$. 
Let $X$ be a smooth projective variety, $\pi_1(X)\to \Gamma$
a quotient and $\tilde X_{\Gamma}\to X$
 the corresponding $\Gamma$-cover.
Then the following are equivalent.
\begin{enumerate}
\item  $\tilde X_{\Gamma}$ is  biholomorphic to a  
semialgebraic subset of a projective variety.
\item $X$  is a  fiber bundle over $\bd/\Gamma$.
\item $\tilde X_{\Gamma}$  is biholomorphic to a 
product $\bd \times F$ where 
 $F$ is a projective  variety.
\end{enumerate}
\end{thm}

Note that in (\ref{main.thm.quot-version}.1)
 we do not assume that the $\Gamma$-action is
semialgebraic on $\tilde X_{\Gamma}$. As noted in 
\cite[3.1]{chk}, this is usually not true, but it holds
if  $\tilde X_{\Gamma}$ is  biholomorphic to a  bounded semialgebraic 
domain in $\c^n$ \cite{zai}.

\begin{rem} Small changes are needed in the statement of
Theorem \ref{main.thm.quot-version} if 
 $\bd$ is the unit disc in $\c$.
By \cite{siu87, MR1098610}, in this case there is a morphism
with connected fibers 
$g:X\to C$ to a smooth curve $C$ of genus $\geq 2$, but
$\pi_1(C)$ may be larger than $\Gamma$. 
The rest of the proof works if we replace
  $\tilde X_{\Gamma}$ by  the cover corresponding to $\pi_1(C)$.
\end{rem}

\begin{rem} Every biholomorphism of  $\bd\times \c^m \times F$
preserves the projections
$$
\bd\times \c^m \times F \to \bd\times \c^m  \to \bd.
$$
Thus, if the universal cover of $X$ is
$ \bd\times \c^m \times F$, then
$\pi_1(X)$ acts on $\bd\times \c^m$ and on $ \bd$
properly discontinuously but not necessarily  freely.
By passing to a finite cover $X'\to X$, we may assume that
the $\pi_1(X')$-action on   $ \bd$ is also free, modulo its kernel 
(cf.\ \cite[p.154]{MR0130324}).
 Thus if we quotient by $\pi_1(X')$,
 we obtain
morphisms of algebraic varieties
$X'\to Y'\to Z'$ such that
\begin{enumerate}
\item $X'\to Y'$ is a fiber bundle whose fiber is a
normal, projective variety,
\item $Y'\to Z'$ is a smooth morphism whose fibers are Abelian varieties
(but it need not be a holomorphic fiber bundle) and
\item $Z'=\bd/\Gamma'$ for some group $\Gamma'$ acting
properly discontinuously and freely.
\end{enumerate}
\end{rem}

\begin{say}[Main steps of the proof] \label{main.step.say}

A quotient  $\sigma: \pi_1(X)\to \Gamma$ is equivalent 
to a homotopy class of continuous maps
$X\to  \bd/\Gamma$. 

First, we use the  purely topological results of Section \ref{top.sect} to
prove that every map in this  homotopy class  is surjective.

Second, a combination of \cite{eel-sam} and \cite{siu}
implies that  there is a holomorphic map
$g_{\sigma}:X\to \bd/\Gamma$  that induces $\sigma$
on the fundamental groups; see (\ref{bunded.defn}) and 
(\ref{bunded.thm}) for details.
These two steps so far used only that $\tilde X_{\Gamma}$  is
homotopic to a finite CW complex.

The third step, which is modeled on
\cite[Sec.2]{chk},  is best explained  by the special case
when $Y:=\mathbb D/\Gamma$ is a curve, $X$ is a surface and 
 $g:X\to Y$ has connected fibers.

Fix a  normal, projective surface $\bar X$ that contains
$\tilde X$ as an open  semialgebraic subset.
Let $\tilde g:\tilde X\to \tilde Y$ be the lifting of $g$ and
$\tilde F$  a 
general fiber of
$\tilde g$. Note that $\tilde F$ has self-intersection 0
and it moves, thus it determines a morphism
$\bar g:\bar X\to \bar Y$. 

The following two observations almost contradict each other.
\begin{enumerate}
\item[(i)] $\bar g:\bar X\to \bar Y$ is an algebraic family of curves,
hence it gives an algebraic morphism of $\bar Y$ to the moduli space
of curves. The fibers of this moduli map are algebraic subsets
of $\bar Y$.
\item[(ii)] Each fiber of  $g:X\to Y$ gives rise to
$|\Gamma|$ copies of itself among the fibers of $\tilde g$.
Thus the moduli map has infinite fibers.
\end{enumerate}
The only way out is if the moduli map is constant
and all the smooth fibers of  $g:X\to Y$ are isomorphic to each other.
Note that any singular fiber of $g$ would lead to
infinitely many singular fibers of $\tilde g$ and of $\bar g$;
this is again impossible.
Thus there are no singular fibers and  $g:X\to Y$ is a  fiber bundle.
\end{say}

\begin{say}[Bounded symmetric domains]\label{bunded.defn}
For general reference, see \cite[Chap.VIII]{hel}.

Let $\bd$ be a bounded symmetric domain
and $\Gamma$  a group acting properly discontinuously and freely
with compact quotient  $\bd/\Gamma$.

Let $X$ be a smooth projective variety.
A group homomorphism $\sigma:\pi_1(X)\to \Gamma$ is equivalent
to a homotopy class of continuous maps
$X\to \bd/\Gamma$.
By Eells-Sampson \cite{eel-sam}, every  
homotopy class contains a harmonic map
$g_{\sigma}: X\to \bd/\Gamma$.
A theorem of Siu \cite[Thm.6.7]{siu} says that if $\bd$ is
irreducible and the 
rank of $dg_{\sigma}$ is large enough at some point, then
$g_{\sigma}$ is
either holomorphic or conjugate holomorphic.
The result applies whenever $g_{\sigma}$ is surjective and $\dim \bd\geq 2$,
which is assured by the purely topological Theorem \ref{alwayssurjective}.
 Thus, for a suitable choice
of the $\Gamma$-action on $\bd$ we may assume that
 $g_{\sigma}$ is  holomorphic.

We can summarize  the above considerations
as follows.
\end{say}

\begin{lem} \label{bunded.thm}
Let $\Gamma$ be a  group acting properly discontinuously and freely
with compact quotient  $\bd/\Gamma$ on an
irreducible, bounded, symmetric domain $\bd$ of dimension $\geq 2$. 
Let $X$ be a smooth projective variety, $\sigma:\pi_1(X)\to \Gamma$
a quotient and $\tilde X_{\Gamma}\to X$
 the corresponding $\Gamma$-cover.
Assume that $\tilde X_{\Gamma}$  is
homotopic to a finite CW complex.

Then there is a holomorphic map
$g_{\sigma}:X\to \bd/\Gamma$  that induces $\sigma$
on the fundamental groups. \qed 
\end{lem}

\section{Maps to compact $K(\pi,1)$ spaces}
\label{top.sect}

In this section we consider the following.

\begin{ques}\label{basic.top.ques}
Let $X$ be a finite CW complex which is a $K(\pi,1)$ and $A$  a 
compact metric space mapping to $X$.  Let $\tilde X$ be the universal cover
 of $X$ and 
 $\tilde A$  the corresponding cover of $A$ (that is, the fibered product,
 or pull-back). We have a commutative fiber product diagram:
$$
\begin{array}{ccc}
\tilde A & \to & A\\
\tilde\varpi \downarrow \hphantom{\tilde\varpi}&& \hphantom{\varpi}\downarrow\varpi \\
\tilde X & \to & X.
\end{array}
$$
Suppose that $\tilde A$  has some finiteness properties. 
What can one conclude about $A\to X$?
\end{ques}

We get quite strong results if  $\tilde A$ is  homotopy equivalent to a 
 compact metric space 
(Theorems \ref{alwayssurjective}, \ref{homologymap}, 
\ref{genericallyfinitecase}), but we are not sure what the optimal
conclusions should be.
Most of the results in this section apply if  $\tilde X$ is contractible
 and $X=\tilde X/\Gamma$ for 
a properly discontinuous cocompact action of $\Gamma$ 
(that is, we do not need the action to be free).

\begin{defn}
Suppose $X$ is a metric space with the property that closed metric balls are compact.  Let $\pi\subset\operatorname{Isom}(X)$ be a group 
whose action is cocompact and properly discontinuous.  Denote its action by $T_\alpha:X\to X$, so $T_\alpha T_\beta=T_{\alpha\beta}$.  We 
call such a pair $(X,\pi)$ a \emph{proper $\pi$-space}.  If we additionally fix a set $\{p_\alpha\}_{\alpha\in\pi}$ so $T_\alpha p_\beta=p_{\alpha\beta}$, 
then we call such a triple $(X,\pi,\{p_\alpha\}_{\alpha\in\pi})$ a \emph{pointed proper $\pi$-space}.
\end{defn}

For any pointed proper $\pi$-space $(X,\pi,\{p_\alpha\}_{\alpha\in\pi})$, we may rescale the metric so that the closed unit balls of radius $1$ 
centered at the $\{p_\alpha\}_{\alpha\in\pi}$ cover all of $X$.  All of the results in this section are indifferent to such rescalings.  Therefore 
we shall often assume implicitly that the particular pointed proper $\pi$-space in question has this property.

\begin{lem}\label{Wfunc}
Let $(X,\pi,\{p_\alpha\}_{\alpha\in\pi})$ be a pointed proper $\pi$-space.  Let $F^t:X\to X$ for $t\in[0,1]$ 
be a homotopy between $F^0=\operatorname{id}_X$ and $F^1:X\to K\subseteq X$ where $K$ is compact.

We denote by $\pi !$ the set of total orders on $\pi$.  For any finite subset $P\subset\pi$, consider the continuous function 
$W_F:[0,1]^P\times\pi !\times X\to X$ defined by:
$$
W_F(\{t_\alpha\}_{\alpha\in P},\sigma,x):=
\Bigl(\prod_{\alpha\in P, \text{ordered by }\sigma}
T_\alpha\circ F^{t_\alpha}\circ T_\alpha^{-1}\Bigr)(x)
\eqno{(\ref{Wfunc}.1)}
$$
For $P_1\subseteq P_2$, map $[0,1]^{P_1}\to[0,1]^{P_2}$ by extending by zero.  Since $F^0$ is the identity, the functions 
$\{W_F\}_{P\subset\pi}$ are compatible with the directed system $\{[0,1]^P\times\pi !\times X\}_{P\subset\pi}$, 
and so we really have a continuous function defined on the direct (inductive) limit, which we also denote by $W_F$.

We have that for every $N<\infty$ there exists $M<\infty$ such that:
\begin{enumerate}\setcounter{enumi}{1}
\item 
$\sup\limits_{t_\alpha\ne 0}d(p_\alpha,x)<N \implies d(W_F(\{t_\alpha\}_{\alpha\in\pi},\sigma,x),x)<M$ and
\item 
$t_\gamma=1\text{ and }\sup\limits_{t_\alpha\ne 0}d(p_\alpha,p_{\gamma})<N \implies d(W_F(\{t_\alpha\}_{\alpha\in\pi},\sigma,x),p_{\gamma})<M$.
\end{enumerate}
\end{lem}

\begin{proof}
For an order $\sigma\in\pi !$ and an
element $\gamma\in\pi$, define $\sigma^\gamma\in\pi !$ by the property:
$$
\beta_1\prec_{\sigma^\gamma}\beta_2\iff\gamma\beta_1\prec_\sigma\gamma\beta_2
$$
Then we have the following identity:
$$
W_F(\{t_\alpha\}_{\alpha\in\pi},\sigma,x)=T_\gamma W_F(\{t_{\gamma\alpha}\}_{\alpha\in\pi},\sigma^{\gamma},T_\gamma^{-1}x).
\eqno{(\ref{Wfunc}.4)}
$$

To prove (\ref{Wfunc}.2), note that given $x\in X$, we can find $\gamma\in\pi$ such that $d(p_{\gamma},x)\leq 1$.  Then by (\ref{Wfunc}.4) we have:
$$
d(W_F(\{t_\alpha\}_{\alpha\in\pi},\sigma,x),x)=d(W_F(\{t_{\gamma\alpha}\}_{\alpha\in\pi},\sigma^{\gamma},T_{\gamma}^{-1}x),T_{\gamma}^{-1}x)
$$
Given $N<\infty$, let $P_N=\{\alpha\in\pi:d(p_{\alpha},p_1)<N\}$ (this set is finite).  Then the above quantity is certainly 
in the image of the map $[0,1]^{P_N}\times P_N!\times B(p_1,1)\to\mathbb R_{\geq 0}$ which sends 
$(\{t_{\alpha}\}_{\alpha\in P_N},\sigma,y)$ to $d(W_F(\{t_{\alpha}\}_{\alpha\in P_N},\sigma,y),y)$.  
The domain of this function is compact, so its range is bounded above by some $M<\infty$.

To prove (\ref{Wfunc}.3), we observe that by (\ref{Wfunc}.4) we have:
$$
d(W_F(\{t_\alpha\}_{\alpha\in\pi},\sigma,x),p_\gamma)=d(W_F(\{t_{\gamma\alpha}\}_{\alpha\in\pi},\sigma^{\gamma},T_{\gamma}^{-1}x),p_1)
$$
Note that since $t_\gamma=1$, the function corresponding to $1\in\pi$ in the evaluation of $W_F$ on the right hand side has 
image contained in $K$.  Thus the right hand side is in the image of the map $[0,1]^{P_N}\times P_N!\times K\to\mathbb R_{\geq 0}$ which sends 
$(\{t_{\alpha}\}_{\alpha\in P_N},\sigma,y)$ to $d(W_F(\{t_{\alpha}\}_{\alpha\in P_N},\sigma,y),p_1)$.  
The domain of this function is compact, so its range is bounded above by some $M<\infty$.
\end{proof}

The notation $W_F$ will be used throughout this section to denote the function given in (\ref{Wfunc}.1).  Also, the order $\sigma\in\pi !$ 
will be irrelevant from now on.  Thus we assume implicitly that such an order is fixed, and we suppress $\sigma$ from the notation.  
Let us now fix a continuous (cutoff) function $w:\mathbb R_{\geq 0}\to[0,1]$ with $w(t)=1$ for $t\in[0,1]$ and $w(t)=0$ for $t\geq 2$.

\begin{lem}\label{coarsehe}
Let $(X,\pi)$ be a proper $\pi$-space, and suppose that $X$ is contractible.  
Then there exists a continuous function $R:X\times X\times[0,1]\to X$ such that $R(x,y,0)=x$, $R(x,y,1)=y$, and 
for every $N<\infty$ there exists $M<\infty$ such that:
$$
d(x,y)<N\implies d(x,R(x,y,t))<M
\eqno{(\ref{coarsehe}.1)}
$$
\end{lem}

\begin{proof}
Fix $\{p_\alpha\}_{\alpha\in\pi}$ such that $(X,\pi,\{p_\alpha\}_{\alpha\in\pi})$ is a pointed proper $\pi$-space.

Let $G^t:X\to X$ ($t\in[0,1]$) be a contraction to $p_1$, that is $G^0(x)=x$ and $G^1(x)=p_1$.  We define:
$$
R(x,y,t):=\begin{cases}W_G(\{2t\cdot w(d(x,p_\alpha))\},x)&t\in[0,\frac 12]\cr W_G(\{2(1-t)\cdot w(d(x,p_\alpha))\},y)&t\in[\frac 12,1]\end{cases}
$$
For $t=\frac 12$, the two definitions agree, since one of the functions in the composition defining $W_G(\{w(d(x,p_\alpha))\},\cdot)$ 
is constant (since for all $x$, one of the values $w(d(x,p_\alpha))$ is equal to $1$).  Clearly $R(x,y,0)=x$ and $R(x,y,1)=y$.  
By (\ref{Wfunc}.2), we have the desired property (\ref{coarsehe}.1).
\end{proof}

\begin{prop}\label{keyhe}
Notation and assumptions as in  (\ref{basic.top.ques}).
Suppose that we have a homotopy equivalence $\rho:\tilde A\to Z$ for some compact metric space $Z$.  Then the 
product map $(\rho,\tilde\varpi):\tilde A\to Z\times\tilde X$ is a proper homotopy equivalence.
\end{prop}

\begin{proof}
Fix a metric on $\tilde X$ and points $\{p_\alpha\}_{\alpha\in\pi_1(X)}$ so that $(\tilde X,\pi_1(X),\{p_\alpha\}_{\alpha\in\pi_1(X)})$ is 
a pointed proper $\pi_1(X)$-space.  Metrize $\tilde A$ by $d_{\tilde A}=d_A+d_{\tilde X}$; then $\tilde A$ is also a proper $\pi_1(X)$-space, 
and the quotient metric induces the same topology as $d_A$.

Denote by $\rho':Z\to\tilde A$ the function paired with $\rho$ forming the homotopy equivalence.  Let $F^t:\tilde A\to\tilde A$ ($t\in[0,1]$) 
be a homotopy between the identity map and $\rho'\circ\rho$.  Thus for all $a\in\tilde A$, 
we have $F^0(a)=a$ and $F^1(a)\in K$ for some compact set $K=\operatorname{im}\rho'$.

Define a map $\phi:Z\times\tilde X\to\tilde A$ as follows:
$$
\phi(z,x):=W_F(\{w(d(x,p_\alpha))\},\rho'(z))
$$
By (\ref{Wfunc}.3), $\phi:Z\times\tilde X\to\tilde A$ satisfies $\sup_{(z,x)\in Z\times\tilde X}d(x,\tilde\varpi(\phi(z,x)))<\infty$.  
Thus $\phi$ is proper.  Since $A$ is compact, the map $(\rho,\tilde\varpi):\tilde A\to Z\times\tilde X$ is proper as well.  
We now prove both compositions are proper homotopic to the identity.

Consider first the function $(\rho,\tilde\varpi)\circ\phi:Z\times\tilde X\to Z\times\tilde X$, which is given by:
$$
(\rho(W_F(\{w(d(x,p_\alpha))\},\rho'(z))),\tilde\varpi(\phi(z,x)))
\eqno{(\ref{keyhe}.1)}
$$
Using (\ref{coarsehe}), the above function  is proper homotopic via:
$$
(\rho(W_F(\{w(d(x,p_\alpha))\},\rho'(z))),R(\tilde\varpi(\phi(z,x)),x,t))
$$
to:
$$
(\rho(W_F(\{w(d(x,p_\alpha))\},\rho'(z))),x)
\eqno{(\ref{keyhe}.2)}
$$
Now since $Z$ is compact, the family $(\rho(W_F(\{t\cdot w(d(x,p_\alpha))\},\rho'(z))),x)$ is a proper homotopy between 
(\ref{keyhe}.2) and $(\rho(\rho'(z)),x)$, which by definition is proper homotopic to $(z,x)$, i.e.\ the identity map.

Now consider the function $\phi\circ(\rho,\tilde\varpi):\tilde A\to\tilde A$, which is given by:
$$
W_F(\{w(d(\tilde\varpi(a),p_\alpha))\},\rho'(\rho(a)))
\eqno{(\ref{keyhe}.3)}
$$
We know $\rho'(\rho(a))$ is homotopic to the identity map.  Thus (\ref{keyhe}.3)
 is homotopic to:
$$
W_F(\{w(d(\tilde\varpi(a),p_\alpha))\},a)
\eqno{(\ref{keyhe}.4)}
$$
and the homotopy is proper by (\ref{Wfunc}.3).  Now $W_F(\{t\cdot w(d(\tilde\varpi(a),p_\alpha))\},a)$ 
gives a homotopy between (\ref{keyhe}.4)
 and the identity map; this homotopy is proper by (\ref{Wfunc}.2).
\end{proof}

We shall have occasion below to use various flavors of (usually singular) homology and cohomology.  We let $H_\ast$ and $H^\ast$ denote standard
homology and cohomology; these are functorial with respect to homotopy classes of continuous maps.  We let $H_\ast^{\operatorname{lf}}$ and 
$H^\ast_c$ denote locally finite homology and compactly supported cohomology respectively; these are functorial with respect to proper homotopy classes of proper maps 
(see \cite[III.10--12]{homologybook} for some definitions and basic properties).  We use $\tilde H_\ast$ to denote reduced homology.  Our coefficient group is always $\mathbb Z$.

\begin{thm}\label{alwayssurjective}
Notation and assumptions as in  (\ref{basic.top.ques}).
Assume further that either $X$ is a manifold or that $X$ is a 
(possibly singular) complex analytic space.  

If $\tilde A$ is homotopy equivalent to a nonempty
compact metric space, then $A\to X$ is surjective.
\end{thm}

\begin{proof}
Let us call the homotopy equivalence $\rho:\tilde A\to Z$.  
Pick any point $x\in\tilde X$ where $\tilde X$ is locally a manifold of dimension $n$, and suppose that $x$ is not in the image of $\tilde A$.  Then we 
have the following commutative diagram:
$$
\begin{array}{ccc}
\tilde A & \stackrel{\rho\times\tilde\varpi}{\longrightarrow} & 
Z\times\tilde X\\
\tilde\varpi  \downarrow \hphantom{\tilde\varpi}&& \hphantom{p_{\tilde X}}\downarrow p_{\tilde X}\\
\tilde X-\{x\}& \to & \tilde X
\end{array}
$$
By (\ref{keyhe}), the map $\tilde A\to Z\times\tilde X$ is a proper homotopy equivalence, and thus induces an isomorphism on $H_\ast^{\operatorname{lf}}$.  
Clearly the map $Z\times\tilde X\to\tilde X$ is surjective on $H_\ast^{\operatorname{lf}}$.  Thus the map $\tilde A\to\tilde X$ is surjective on $H_\ast^{\operatorname{lf}}$.  On the other hand, 
certainly $H_n^{\operatorname{lf}}(\tilde X-\{x\})\to H_n^{\operatorname{lf}}(\tilde X)$ is not surjective: in the case $X$ is a manifold, 
the fundamental class $[\tilde X]\in H_n^{\operatorname{lf}}(\tilde X)$ is not in the image, and in the case $X$ is a complex analytic 
space, the fundamental class of the irreducible component of $\tilde X$ containing $x$ is not in the image.  This is a contradiction, 
so we conclude that $x$ is in the image of $\tilde A\to\tilde X$.

Thus every point $x\in X$ where $X$ is locally a manifold is in the image of $A\to X$.  If $X$ is a manifold, we are done.  If $X$ is a complex 
analytic space, observe that the set of points where $X$ is a manifold is dense.  On the other hand, $A$ is compact, so the image of $A\to X$ is 
closed.  Thus we are done in this case as well.
\end{proof}

\begin{thm}\label{homologymap}
Notation and assumptions as in  (\ref{basic.top.ques}).
Assume further that $A\to X$ is a complex analytic map of
compact complex analytic spaces and that $X$ is normal.
Suppose that $\tilde A$ is homotopy equivalent to a finite CW complex.  

Then for every fiber  $\tilde A_x$ of $\tilde A$, 
there is a natural map $H_\ast(\tilde A)\to H_\ast(\tilde A_x)$ whose composition with $H_\ast(\tilde A_x)\to H_\ast(\tilde A)$ is the identity.
\end{thm}

\begin{proof}
Let $U$ be the set of points $x\in\tilde X$ where $\tilde X$ is locally a $C^\infty$-manifold over which $\tilde A\to\tilde X$ is a 
$C^\infty$-bundle.  Then $\tilde X\setminus U$ is a closed subvariety.  Suppose we prove the theorem for all $x\in U$; then the general 
case is done as follows.  Given any fiber $\tilde A_x$, we can find an open neighborhood which deformation retracts onto it.  Such an 
open neighborhood contains $\tilde A_u$ for all $u\in U\cap N_\epsilon(x)$.  Since $\tilde X$ is normal, $U$ is locally connected near 
$u$.  Thus we get a well defined map $H_\ast(\tilde A)\to H_\ast(\tilde A_u)\to H_\ast(\tilde A_x)$ with the desired property.  Now let 
us prove the result for $x\in U$.

Let $\rho:\tilde A\to Z$ be the homotopy equivalence from $\tilde A$ to a finite CW complex $Z$, and let $\phi:Z\times\tilde X\to\tilde A$ be 
the resulting proper homotopy equivalence from (\ref{keyhe}).  Let $2n$ be the dimension of $\tilde X$ as a 
real manifold, and define $H_\ast(\tilde A)\to H_\ast(\tilde A_x)$ as follows:
\begin{equation*}
H_\ast(\tilde A)\xrightarrow{\rho_\ast}H_\ast(Z)\xrightarrow{\times[\tilde X]}H_{\ast+2n}^{\operatorname{lf}}(Z\times\tilde X)
\xrightarrow{\phi_\ast}H_{\ast+2n}^{\operatorname{lf}}(\tilde A)\to H_{\ast+2n}(\tilde A,\tilde A\setminus\tilde A_x)=H_\ast(\tilde A_x)
\end{equation*}
By construction, the maps for different nearby values of $x$ are compatible.  Now let us prove that the composition $H_\ast(\tilde A)\to H_\ast(\tilde A_x)\to H_\ast(\tilde A)$ 
is the identity map.

Let us first consider the last few maps in this long composition, namely:
\begin{equation*}
H_{\ast+2n}^{\operatorname{lf}}(\tilde A)\to H_{\ast+2n}(\tilde A,\tilde A\setminus\tilde A_x)=H_\ast(\tilde A_x)\to H_\ast(\tilde A)
\end{equation*}
This map is just given by the cap product with $\tilde\varpi^\ast([[\tilde X]])\in H^{2n}_c(\tilde A)$, where we denote 
by $[[\tilde X]]\in H^{2n}_c(\tilde X)$ the unique class whose evaluation on the fundamental class of $\tilde X$ is $1$.  
Thus we conclude that the composition $H_\ast(\tilde A)\to H_\ast(\tilde A_x)\to H_\ast(\tilde A)$ is:
\begin{equation*}
H_\ast(\tilde A)\xrightarrow{\rho_\ast}H_\ast(Z)\xrightarrow{\times[\tilde X]}H_{\ast+2n}^{\operatorname{lf}}(Z\times\tilde X)\xrightarrow{\phi_\ast}H_{\ast+2n}^{\operatorname{lf}}(\tilde A)\xrightarrow{\frown\tilde\varpi^\ast([[\tilde X]])}H_\ast(\tilde A)
\end{equation*}
Note that $\alpha\in H_\ast(Z)$ is sent to:
\begin{align*}
\phi_\ast(\alpha\times[\tilde X])\frown\tilde\varpi^\ast([[\tilde X]])&=\phi_\ast((\alpha\times[\tilde X])\frown\phi^\ast\tilde\varpi^\ast([[\tilde X]]))\cr
&=\phi_\ast((\alpha\times[\tilde X])\frown p_{\tilde X}^\ast([[\tilde X]]))\cr
&=\phi_\ast(\alpha\times\{x\})
\end{align*}
For this, we used the fact from (\ref{keyhe}) that $\tilde\varpi\circ\phi$ is proper homotopic to $p_{\tilde X}:Z\times\tilde X\to\tilde X$.
Thus our map is just $H_\ast(\tilde A)\xrightarrow{\rho_\ast}H_\ast(Z)\xrightarrow{(z\mapsto(z,x))_\ast}H_\ast(Z\times\tilde X)\xrightarrow{\phi_\ast}H_\ast(\tilde A)$, 
which is clearly the identity map.
\end{proof}

\begin{thm}\label{genericallyfinitecase}
Notation and assumptions as in  (\ref{basic.top.ques}).
Assume further that $A\to X$ is a generically 
finite complex analytic map of
compact complex analytic spaces and that $X$ is normal.

Suppose that $\tilde A$ is homotopy equivalent to a finite CW complex. 
 Then $A\to X$ is a bundle 
with finite fiber.
\end{thm}

\begin{proof}
Denote by $Z$ the finite CW complex homotopy equivalent to $\tilde A$.  Then $Z$ has finitely many connected components, corresponding to 
the connected components of $\tilde A$.  We have a representation $\pi_1(X)\to\operatorname{Sym}(\pi_0(\tilde A))=\operatorname{Sym}(\pi_0(Z))$.  
It suffices to prove the theorem in the case where we replace $X$ with the finite covering corresponding to the kernel of this representation.  
In this case, the connected components of $A$ correspond exactly to the connected components of $\tilde A$, and it suffices to prove the theorem 
for each of these separately.  Thus we may assume without loss of generality that $A$ and $\tilde A$ are connected.  Thus $Z$ is connected as well.

By (\ref{homologymap}), we know that the homology of any fiber of $\tilde A\to\tilde X$ surjects onto $H_\ast(\tilde A)=H_\ast(Z)$.  The former can be just a finite 
number of points, so we conclude that $H_i(Z)=0$ for $i>0$.  Since $Z$ is connected, we have $H_0(Z)=\mathbb Z$.

Fix a CW structure on $\tilde X$.  On the chain level (for cellular cohomology), we have $C_c^\ast(Z\times\tilde X)=C^\ast(Z)\otimes C_c^\ast(\tilde X)$.  
By the universal coefficient theorem, we have $H^0(Z)=\mathbb Z$ and $H^i(Z)=0$ 
for $i>0$.  Thus the K\"unneth formula (see \cite[Thm.9.16]{MR1757274})
implies there is an isomorphism $H_c^\ast(\tilde X)\to H_c^\ast(Z\times\tilde X)=H_c^\ast(\tilde A)$.  
Thus we conclude that $\dim\tilde A=\dim\tilde X$ and that $\tilde A$ is irreducible ($\tilde X$ is irreducible since it is normal).

Now let us show that $Z$ is simply connected.  Suppose $(\mathbb S^1,\mathbf 0)\to(Z,z)$ is any map of pointed spaces representing $\alpha\in\pi_1(Z,z)$.  Then we get 
a corresponding map $f:\mathbb S^1\times\tilde X\to Z\times\tilde X\to\tilde A$ (the latter map being the proper homotopy equivalence from (\ref{keyhe})).  
Pick a point $a\in\tilde A$ where $\tilde A$ is a manifold, and smooth $f$ in a neighborhood of $f^{-1}(\{a\})$, so that (for a generic smoothing) $f^{-1}(\{a\})$ is a disjoint 
union of some number of circles, say $C_1,\ldots,C_k$, each of which is transverse to $\mathbf 0\times\tilde X$.  Let $[\tilde A]$ denote the fundamental 
class of $\tilde A$, so that the natural map on $H_\ast^{\operatorname{lf}}$ (induced from $\tilde X\to Z\times\tilde X\to\tilde A$) 
sends $[\tilde X]$ to $[\tilde A]$.  Now we may write:
$$
1=\langle\{a\},[\tilde A]\rangle=\langle\{a\},f(\mathbf 0\times\tilde X)\rangle=\langle f^{-1}(\{a\}),\mathbf 0\times\tilde X\rangle=\sum_i\langle C_i,\mathbf 0\times\tilde X\rangle
$$
Each term $\langle C_i,\mathbf 0\times\tilde X\rangle$ is just the winding number of $C_i$ around $\mathbb S^1$.  Equip each $C_i$ which intersects 
$\mathbf 0\times\tilde X$ with a base point on this intersection.  Then each such $C_i$ is represents the element $\alpha^{\langle C_i,\mathbf 0\times\tilde X\rangle}$ 
in $\pi_1(Z,z)$.  On the other hand, $C_i$ is mapped to $\{a\}$ in $\tilde A$, and thus is trivial in $\pi_1(\tilde A)=\pi_1(Z)$.  Thus 
$\alpha^{\langle C_i,\mathbf 0\times\tilde X\rangle}$ is trivial in $\pi_1(Z,z)$.  Taking the product over all $C_i$ with 
nonzero $\langle C_i,\mathbf 0\times\tilde X\rangle$ and observing that the sum of these values is $1$, we conclude that $\alpha\in\pi_1(Z,z)$ is trivial.  
Thus $Z$ is simply connected.

Since $Z$ is simply connected, and $H_i(Z)=0$ for $i>0$, we conclude that $Z$ is contractible.  Thus $\tilde A\to\tilde X$ is a homotopy equivalence.  
From this, we see that the map $A\to X$ induces an isomorphism on $\pi_\ast$, and thus is a homotopy equivalence as well.  Now this implies 
that $\varpi_\ast([A])=[X]$, so in particular the general fiber of $A\to X$ is a single point.

If $A$ is projective and $\varpi^{-1}(\{x\})$ is positive dimensional, then
$\bigl[\varpi^{-1}(\{x\})\bigr]$ is a nonzero class in
$H_{\ast}(A)$ that gets killed in $H_{\ast}(X)$. This is a
contradiction, hence $A\to X$ is an isomorphism.

In the general case, we argue as follows.
Let $S\subset X$ be the set of $x\in X$ whose fiber $\varpi^{-1}(\{x\})$ is not a single point.  Then $S$ is a closed subvariety of $X$.  
Suppose that $S$ is nonempty, and consider the natural morphism of long exact sequences:
$$
\begin{array}{c@{\hspace{3pt}}c@{\hspace{3pt}}c@{\hspace{3pt}}c@{\hspace{3pt}}c@{\hspace{3pt}}c@{\hspace{3pt}}c@{\hspace{3pt}}c@{\hspace{3pt}}c}
H_{\ast+1}(A)& \to & H_{\ast+1}(A,\varpi^{-1}(S))& \to & H_\ast(\varpi^{-1}(S))& \to & H_\ast(A)& \to & H_\ast(A,\varpi^{-1}(S))\\
 \downarrow && \downarrow && \downarrow && \downarrow && \downarrow\\
H_{\ast+1}(X)& \to & H_{\ast+1}(X,S)& \to & H_\ast(S)& \to & H_\ast(X)& \to & H_\ast(X,S)
\end{array}
$$
A closed subvariety of a complex analytic space has an open neighborhood which deformation retracts onto it.  
Hence we can replace the relative homologies with the (reduced) homologies of the quotient spaces:
$$
\begin{array}{c@{\hspace{3pt}}c@{\hspace{3pt}}c@{\hspace{3pt}}c@{\hspace{3pt}}c@{\hspace{3pt}}c@{\hspace{3pt}}c@{\hspace{3pt}}c@{\hspace{3pt}}c}
H_{\ast+1}(A)& \to & \tilde H_{\ast+1}(A/\varpi^{-1}(S))& \to & H_\ast(\varpi^{-1}(S))& \to & H_\ast(A)& \to & \tilde H_\ast(A/\varpi^{-1}(S))\cr
 \downarrow && \downarrow && \downarrow && \downarrow && \downarrow\\
H_{\ast+1}(X)& \to & \tilde H_{\ast+1}(X/S)& \to & H_\ast(S)& \to & H_\ast(X)& \to & \tilde H_\ast(X/S)
\end{array}
$$
The map $A\to X$ is a homotopy equivalence, so the induced maps on $H_\ast$ are isomorphisms.  The map $A/\varpi^{-1}(S)\to X/S$ 
is a homeomorphism, so the induced maps on $H_\ast$ are isomorphisms.  Thus by the five lemma, the map $H_\ast(\varpi^{-1}(S))\to H_\ast(S)$ 
is an isomorphism.  However, $X$ is normal so Zariski's main theorem implies that $\dim\varpi^{-1}(S)>\dim S$, contradicting the 
isomorphism $H_\ast(\varpi^{-1}(S))\overset\sim\to H_\ast(S)$.  Hence $S$ is empty, that is $A\to X$ is an isomorphism of complex spaces.
\end{proof}

\section{Semialgebraic covers}\label{sa.sect}

\begin{defn}[Semialgebraic sets] \label{semialg.defns}
See  \cite[Chap.2]{bcr} for  a detailed treatment
and for the results that we use.

A  {\it basic open semialgebraic subset} of $\r^n$ is defined by
finitely many polynomial inequalities $g_i(x_1,\dots, x_n)> 0$. 
Using finite intersections and
complements we get all semialgebraic subsets.

Let $Y$ be a complex, affine, algebraic variety.
Choose any (not necessarily closed) embedding $Y\subset \c^N$. 
Identifying $\c^N$ with $\r^{2N}$ we get the notion of
semialgebraic subsets of $\c^N$ and of $Y$. 
The latter is independent of the embedding.
Thus we can talk about {\it semialgebraic subsets} of any complex 
algebraic variety. If $f:X\to Y$ is a morphism of varieties
and $W\subset Y$ is semialgebraic, then so is $f^{-1}(W)$.
By the Tarski--Seidenberg theorem, if 
$V\subset X$ is semialgebraic, then so is $f(V)$.

The closure, interior or a connected component of a semialgebraic set is again
semialgebraic. Every compact semialgebraic set has a
semialgebraic triangulation. Thus every 
open semialgebraic subset is homeomorphic to the interior
of a finite polyhedron. In particular, it is
homotopic to a finite CW complex.

An open semialgebraic subset of an algebraic variety
$Y$ is also naturally a complex space.
We say that a complex space $W$ is {\it semialgebraic}
if it is biholomorphic to an  open semialgebraic subset of
a projective variety $Y$. Note that usually $Y$ and the
embedding are very far from being unique in any sense.
\end{defn}

\begin{rem} Gabrielov pointed out that the above
 properties are also shared by sets
definable in an o-minimal theory; see \cite{dri}.
Thus our results have natural analogs in any o-minimal theory.
It would be interesting to find some examples of universal covers
that are o-minimal but
not semialgebraic.
\end{rem}

\begin{defn}[Chow varieties]\label{chow.defn}
See \cite[Secs.X.6--8]{hod-ped} or \cite[Sec.I.3]{rc-book}
for precise definitions and proofs.

Let $Z$ be a projective variety over $\c$.
An effective $r$-cycle is a formal linear combination
$W=\sum_i m_iW_i$ where the $W_i$ are  
irreducible $r$-dimensional subvarieties of $Z$
and the $m_i$ are natural numbers. The homology class of
$W$ is defined as  $[W]:=\sum m_i[W_i]\in  H_{2r}(Z,\z)$.

For a homology class 
$\alpha\in H_{2r}(Z,\z)$,
let $\chow_{\alpha}(Z)$ denote the Chow variety parametrizing
those  effective $r$-cycles whose homology class 
equals $\alpha$. For an $r$-cycle $W$, the corresponding
Chow point  is denoted by $Ch(W)\in \chow_{\alpha}(Z)$.
Then $\chow_{\alpha}(Z)$ is a projective algebraic variety and
there is a universal family
$$
\begin{array}{ccc}
\univ_{\alpha}(Z) & \stackrel{\pi}{\to} & Z\\
u\downarrow\hphantom{u} && \\
\chow_{\alpha}(Z).
\end{array}
$$
Fix next a normal, projective variety $F$ and consider the
subset $\chow_{(F,\alpha)}(Z)\subset \chow_{\alpha}(Z)$
parametrizing the images of embeddings $\tau:F\into Z$ such that
$[\tau(F)]=\alpha$. Unfortunately, $\chow_{(F,\alpha)}(Z)$
need not be algebraic if $F$ has an infinite discrete
group of automorphisms.
We can, however, easily remedy this problem.
Fix ample divisors  $L$ (resp.\ $H$) on $Z$ (resp.\ on  $F$)
and a number $C>0$. 
We can then look at  the images of embeddings $\tau:F\into Z$ such that
$[\tau(F)]=\alpha$ and the intersection numbers
$\bigl(L^i\cdot \tau^* H^j\bigr)_{F}$ are
$\leq C$  for $i+j=\dim F$. 
(Note that this is essentially equivalent to bounding
the degree, and thus the homology class, of the graph
of $\tau$ in $F\times Z$ under the product polarization given by
$H$ and $L$.) 
These form a constructible
algebraic subset   $\chow^{\circ}_{(F,H,C,\alpha)}(Z,L)\subset \chow_{\alpha}(Z)$.

In order to avoid working with constructible sets,
let  $\chow_{(F,H,C,\alpha)}(Z,L)$ denote the closure of
$\chow^{\circ}_{(F,H,C,\alpha)}(Z,L)$ in $ \chow_{\alpha}(Z)$.
There is a universal family
$$
\begin{array}{ccc}
\univ_{(F,H,C,\alpha)}(Z,L) & \stackrel{\pi}{\to} & Z\\
u\downarrow\hphantom{u} && \\
\chow_{(F,H,C,\alpha)}(Z,L)
\end{array}
$$ 
where $u$ is a fiber bundle with fiber $F$ over
$\chow^{\circ}_{(F,H,C,\alpha)}(Z,L)$.
\end{defn}

The aim of this section is to prove the following.

\begin{thm}\label{alg.maps.to.K.pi.1.thm}
 Let $Y$ be a smooth projective variety
with  universal cover $\tilde Y$. Assume that $\tilde Y$ is
contractible. 
Let $g:X\to Y$ be a 
morphism from a  normal, projective variety $X$ to $Y$
and  $\tilde X:=\tilde Y\times_YX$ the corresponding $\pi_1(Y)$-cover.
 If  $\tilde X$ is open, semialgebraic  in a projective variety $\bar X$
then  $g:X\to Y$
is a fiber bundle.
\end{thm}

We follow the  arguments
in (\ref{main.step.say}) but there are some complications.

First, the map $\tilde g$ need not extend to $\bar X$; see (\ref{cubic.exmp}). 
We thus pass from $\bar X$ to $\chow_{\alpha}(\bar X)$
where $\alpha:=[\tilde F]$
 is the homology class of a general fiber of $\tilde g$.

Second, there are no sensible moduli spaces for higher dimensional
varieties in general, but  
$\chow_{(F,H,C,\alpha)}(\bar X,\bar L)$ acts as
 a fiber of the (nonexistent) moduli map.
Applying  (\ref{alg.maps.to.K.pi.1.lem})
to (roughly) the intersection 
 $\tilde X\cap \pi\bigl(\univ_{(F,H,C,\alpha)}(\bar X,\bar L)\bigr)$
 we conclude that 
$\pi: \univ_{(F,H,C,\alpha)}(\bar X,\bar L)\to \bar X$ is onto,
hence the  fibers of  $g:X\to Y$ are isomorphic to each other,
at least over a Zariski open set.

Finally, in order to deal with the singular fibers of $g$,
we prove that
one can factor 
$\tilde g: \tilde X\to \tilde Y'\to \tilde Y$
where $\tilde X\to \tilde Y'$ is a fiber bundle and
$\tilde Y'\to \tilde Y$ is generically finite.
Then we apply (\ref{homologymap}) to
$\tilde Y'/\Gamma\to Y$ to conclude that 
$ \tilde Y'\to \tilde Y$ is in fact an isomorphism.
\medskip

The next simple consequence of (\ref{alwayssurjective}) will be very useful.

\begin{lem}\label{alg.maps.to.K.pi.1.lem}
Notation and assumptions as in (\ref{alg.maps.to.K.pi.1.thm}).
Fix a normal, projective variety $\bar X$ that contains
$\tilde X$ as an open  semialgebraic subset.
Let $\tilde Z\subset\tilde X$ be a  nonempty closed,  $\Gamma$-invariant,
{\em  analytic} subset that is semialgebraic in $\bar X$.

Then $g: \tilde Z/\Gamma \to Y$ is surjective.
\end{lem}

Proof. Note that $\tilde Z/\Gamma$ is a closed analytic subset
of the projective variety $X$, hence algebraic by Chow's  theorem.
Thus (\ref{alwayssurjective}) applies to $g: \tilde Z/\Gamma \to Y$,
hence  $g: \tilde Z/\Gamma \to Y$ is surjective.\qed

\begin{say}[Proof of (\ref{alg.maps.to.K.pi.1.thm})]
\label{pf.alg.maps.to.K.pi.1.thm}
The connected components of $\tilde X$ are in one-to-one
correspondence with the cosets of
$\im[\pi_1(X)\to \pi_1(Y)]$ in $\pi_1(Y)$.
Since  $\tilde X$ is semialgebraic, there are only
finitely many cosets. Thus, by passing to a finite cover
of $Y$, we may assume that $\pi_1(X)\onto \pi_1(Y)$ is surjective
and hence $\tilde X$ is connected.
Note that $g$ is surjective by (\ref{alwayssurjective}). 

Let $F\subset X$ be an irreducible component of a general fiber of
$g$ and let $\alpha:=[F]\in H_*\bigl( X,\z\bigr)$ be the 
homology class of $ F$. 
Fix $C\gg 1$ and let $H$ be an ample line bundle on $X$;
we use $H$ to denote $H|_F$ as well.
Let
 $\chow^*_{\alpha}( X)$ (resp.\ $\chow^*_{(F,\alpha)}(X)$) denote
the unique irreducible component of
 $\chow_{\alpha}( X)$ (resp.\ $\chow_{(F,H,C,\alpha)}( X, H)$)
that contains the Chow point  $Ch(F)$.
(The notation $\chow^*_{(F,\alpha)}(X)$ indicates that
this is independent of $H$ and $C$.)
By (\ref{chow.defn}),  $\chow^*_{\alpha}( X)$ and
 $\chow^*_{(F,\alpha)}( X)$ are projective varieties and
there are universal families
$$
\begin{array}{ccc}
\univ^*_{\alpha}( X) & \stackrel{ \pi}{\to} &  X\\
 u\downarrow\hphantom{u} &&\hphantom{g}\downarrow g \\
\chow^*_{\alpha}( X) & \to  & Y
\end{array}
\qtq{and}
\begin{array}{ccc}
\univ^*_{(F,\alpha)}( X) & \stackrel{ \pi}{\to} &  X\\
 u\downarrow\hphantom{u} &&\hphantom{g}\downarrow g  \\
\chow^*_{(F,\alpha)}( X) & \to  & Y.
\end{array}
\eqno{(\ref{pf.alg.maps.to.K.pi.1.thm}.1)}
$$
Note that $\pi:\univ^*_{\alpha}( X)\to X$ is birational
and an isomorphism over $y\in Y$ if  $g^{-1}(y)$ is reduced
and of dimension $=\dim F$.
On the other hand, $\pi: \univ^*_{(F,\alpha)}( X) \to X$ is birational
$\Leftrightarrow$  it is surjective $\Leftrightarrow$ there is a Zariski dense
open set $Y^0\subset Y$ such that the irreducible components
of the fibers over $Y^0$ are all isomorphic to $F$. 


We could complete the diagrams  with $Y$ in the lower right corner
since every connected, effective cycle that is homologous to 
$F$ is contained in a  fiber of $g$.

Next we pull everything back to $\tilde Y$ to obtain
diagrams where all objects are analytic spaces and all
morphisms are proper.
$$
\begin{array}{ccc}
U_{\alpha}( \tilde  X) & \stackrel{ \tilde \pi}{\to} &  \tilde  X\\
 \tilde  u\downarrow\hphantom{u} &&\hphantom{g}\downarrow \tilde  g \\
C_{\alpha}( \tilde  X) & \to  &  \tilde Y
\end{array}
\qtq{and}
\begin{array}{ccc}
U_{(F,\alpha)}( \tilde  X) & \stackrel{ \tilde  \pi}{\to} &  \tilde  X\\
 \tilde  u\downarrow\hphantom{u} &&\hphantom{g}\downarrow \tilde  g  \\
C_{(F,\alpha)}( \tilde  X) & \to  &  \tilde Y.
\end{array}
\eqno{(\ref{pf.alg.maps.to.K.pi.1.thm}.2)}
$$
Fix a normal, projective variety $\bar X$ that contains
$\tilde X$ as an open  semialgebraic subset and
 a lifting $\tilde F$ of $F$.
Note that $\tilde F$ is an irreducible component of a 
general fiber of
$\tilde g$. Let $\bar\alpha:=[\tilde F]\in H_*\bigl(\bar X,\z\bigr)$ be the 
homology class of $\tilde F$.

Let $\chow^*_{\bar\alpha}(\bar X)$ denote the irreducible component 
of $\chow_{\bar\alpha}(\bar X)$
that contains  the Chow point  $Ch(\tilde F)$.

Let $\bar L$ be an ample line bundle on $\bar X$. 
Choose $C$ large enough and let 
$\chow^*_{(F,\bar\alpha)}(\bar X)$ 
be the union of {\em all}  irreducible components of
$\chow_{(F,H,C,\bar\alpha)}(\bar X,\bar L)$ 
that contain  the Chow point of  any preimage of $F$.
Since the irreducible components $\tilde F'_y$ of general fibers of
$\tilde g$ are deformations of each other, 
the intersection numbers
$\bigl(\tilde L^i\cdot \tilde H^j\cdot \tilde F'_y\bigr)$ are
independent of $\tilde F'_y$ for $i+j=\dim F$. 
Thus if the Chow point of 
{\em one} preimage of $F$ is in $\chow_{(F,H,C,\bar\alpha)}(\bar X,\bar L)$ 
then so is  the Chow point of {\em every} preimage of $F$. 

Thus $\chow^*_{(F,\bar\alpha)}(\bar X)$
 is an algebraic variety, independent of $C$,  and
there are universal families
$$
\begin{array}{ccc}
\univ^*_{\bar\alpha}(\bar X) & \stackrel{\bar \pi}{\to} & \bar X\\
\bar u\downarrow\hphantom{u} && \\
\chow^*_{\bar\alpha}(\bar X)
\end{array}
\qtq{and}
\begin{array}{ccc}
\univ^*_{(F,\bar\alpha)}(\bar X) & \stackrel{\bar \pi}{\to} & \bar X\\
\bar u\downarrow\hphantom{u} && \\
\chow^*_{(F,\bar\alpha)}(\bar X).
\end{array}
\eqno{(\ref{pf.alg.maps.to.K.pi.1.thm}.3)}
$$
(Note, however, that $\univ^*_{\bar\alpha}(\bar X) {\to} \bar X$
need not be birational and usually 
one can not complete the diagrams  with any $\bar Y$ in the lower right corner;
see Example \ref{cubic.exmp}.)
 
The key step of the proof is the following.

\medskip
{\it Claim  \ref{pf.alg.maps.to.K.pi.1.thm}.4.}
Notation and assumptions as above.
There are natural inclusions
$$
\begin{array}{ccc}
U_{\alpha}( \tilde  X) & \into &\univ^*_{\bar\alpha}(\bar X)\\
\tilde u\downarrow\hphantom{u} &&\bar u\downarrow\hphantom{u}  \\
C_{\alpha}( \tilde  X) & \into &\chow^*_{\bar\alpha}(\bar X)
\end{array}
\qtq{and}
\begin{array}{ccc}
U_{(F,\alpha)}( \tilde  X) & \into &\univ^*_{(F,\bar\alpha)}(\bar X) \\
\tilde u\downarrow\hphantom{u} && \bar u\downarrow\hphantom{u} \\
C_{(F,\alpha)}( \tilde  X) & \into & \chow^*_{(F,\bar\alpha)}(\bar X).
\end{array}
$$
Under these inclusions, the images of 
$U_{\alpha}( \tilde  X), C_{\alpha}( \tilde  X),
U_{(F,\alpha)}( \tilde  X), C_{(F,\alpha)}( \tilde  X)$
are open semialgebraic subsets.
\medskip

Proof. It is enough to prove that  the image of 
$U_{\alpha}( \tilde  X)$ is an open semialgebraic subset 
of $\univ^*_{\bar\alpha}(\bar X)$.

As we noted in (\ref{semialg.defns}), 
 $\pi^{-1}(\bar X\setminus \tilde X)\subset 
\univ^*_{\bar\alpha}(\bar X)$
and the complement of its projection  
$$
W_{\alpha}:=\chow^*_{\bar\alpha}(\bar X)\setminus 
u\bigl(\pi^{-1}(\bar X\setminus \tilde X)\bigr)\subset \chow^*_{\bar\alpha}(\bar X)
$$
 are  semialgebraic.
 Note that $W_{\alpha}$
parametrizes those  cycles $Z\subset \bar X$
that are contained in $\tilde X$ and satisfy $[Z]=\bar \alpha$.
Such a $Z$ is in $U_{\alpha}( \tilde  X)$ iff
its image in $Y$ is 0-dimensional. The latter condition is
invariant under deformations, hence
$U_{\alpha}( \tilde  X)$ is a connected component
of $W_{\alpha}$, hence semialgebraic. \qed
\medskip

We can thus apply (\ref{alg.maps.to.K.pi.1.lem}) to   
$C_{(F,\alpha)}( \tilde  X)  \to    \tilde Y$
 to conclude that the composite
$C_{(F,\alpha)}( \tilde  X)  \to    \tilde Y\to Y$
 is surjective and so is $\chow^*_{(F,\alpha)}( X)\to Y$. 
Thus there is a Zariski dense
open set $Y^0\subset Y$ such that the irreducible components
of the fibers over $Y^0$ are all isomorphic to $F$. 

Let $Z\subset \chow^*_{(F,\bar\alpha)}( \bar  X)$
be the Zariski closure of the complement of
$\chow^{\circ}_{(F,\bar\alpha)}( \bar  X)$.
Then $Z\cap \tilde X$ is a closed analytic subset,
invariant under $\Gamma$, and its image in $Y$
is contained in $Y\setminus Y^0$. Thus, by  (\ref{alg.maps.to.K.pi.1.lem}),
$Z\cap \tilde X$ is empty. Therefore
$U_{\alpha}( \tilde  X) \to 
C_{\alpha}( \tilde  X)$ and $  u:\univ_{\alpha}( X) \to 
\chow_{\alpha}(  X)$
are fiber bundles with fiber $F$.

The set of points where
$\univ^*_{(F,\bar\alpha)}(\bar X)\to \bar X$ is not \'etale is closed,
 its intersection with $\tilde X$ is a closed analytic subset,
invariant under $\Gamma$ and its image in $Y$
is contained in $Y\setminus Y^0$.
Using (\ref{alg.maps.to.K.pi.1.lem}) again
we see that 
 $U_{\alpha}( \tilde  X)\to \tilde X$ and $ \univ_{\alpha}( X) \to X$ 
are isomorphisms.
Thus $\tilde g$ factors as
$$
\tilde g: \tilde X=U_{\alpha}( \tilde  X)\stackrel{\tilde u}{\to} 
C_{\alpha}( \tilde  X) \to \tilde Y.
$$
Here $C_{\alpha}( \tilde  X)$ is semialgebraic and
$C_{\alpha}( \tilde  X)/\Gamma =\chow_{\alpha}(X)$
has a generically finite morphism to $Y$.
By (\ref{homologymap}) we conclude that
$\chow_{\alpha}(X)\to Y$ is an isomorphism. 

Thus $g:X\to Y$ is isomorphic to the fiber bundle 
$  u:\univ_{\alpha}( X) \to 
\chow_{\alpha}(  X)$. \qed

\end{say}

\begin{say}[Proof of Theorem \ref{main.thm.quot-version}] {\ }

If (\ref{main.thm.quot-version}.2) holds then
$\tilde X_{\Gamma}$ is a holomorphic fiber bundle over $\bd$ with fiber $F$.
By \cite{gra}, every fiber bundle over a contractible Stein space is trivial,
giving (\ref{main.thm.quot-version}.3). 

Every bounded symmetric domain is semialgebraic,
hence  (\ref{main.thm.quot-version}.3) implies
(\ref{main.thm.quot-version}.1).

Finally, assume (\ref{main.thm.quot-version}.1).
By Lemma \ref{bunded.thm} we have a holomorphic map
$X\to \bd/\Gamma$ and it is a fiber bundle by
 Theorem \ref{alg.maps.to.K.pi.1.thm}. \qed
\end{say}

\begin{exmp}\label{cubic.exmp}
Let ${\mathbb B}^2\subset \c^2$ denote the unit ball. Then
$\p^1\times {\mathbb B}^2$ is the universal cover of an
algebraic threefold. It can also be realized as an
open semialgebraic subset of a smooth cubic 3--fold
$\bar X_3\subset \p^4$ such that the
fibers of  $\p^1\times {\mathbb B}^2\to{\mathbb B}^2 $
become lines in $\bar X_3$. The family of all lines on  $\bar X_3$
is irreducible and there are 6 lines through a general point.
Thus the projection $\p^1\times {\mathbb B}^2\to{\mathbb B}^2 $
will not extend to $\bar X_3$ in any way.
\end{exmp}

\section{Open problems}\label{sec.conjectures}

Here we discuss a possible approach to  
 Conjecture  \ref{semi.main.conj} and several other questions that, directly or
indirectly, relate to it.
We also discuss which arguments in \cite{chk} generalize and which need
further new ideas.

The first part of the proof in \cite{chk} shows that if
$\tilde X$ is quasi-projective then
$\pi_1(X)$ has a finite index abelian subgroup.
One can follow the proof given in  \cite{chk}
 until the point where we need to
exclude the case when $X$ is of general type. If
$\tilde X$ is quasi-projective, this is done using 
Kobayashi-Ochiai  \cite{kob-och}. 

In the semialgebraic case there is no contradiction.
In fact, if ${\mathbb U}$ is a bounded open subset of a
Stein manifold and $\Gamma$ is a group acting properly discontinuously
and freely with compact quotient, then ${\mathbb U}/\Gamma$
is a projective variety with ample canonical class.
(This is essentially due to Poincar\'e; see \cite[Chap.5]{shaf-book}
for details.) 
While this sounds very general, there are  few
such examples known aside from bounded symmetric domains \cite{sab}.
The results of \cite{vey, won, fra1} show that under various additional
restrictions, such a ${\mathbb U}$ is necessarily a
 bounded symmetric domain;
see \cite{isa-kra} for a survey of closely related results. 
From our point of view, this leads to
the following problem.

\begin{ques}\label{bdd.dom.quot.then.sum}
  Let ${\mathbb U}$ be a semialgebraic bounded open subset of an
affine variety. Assume that  there is a group $\Gamma$ acting properly 
discontinuously
with compact quotient. Is ${\mathbb U}$ necessarily a
 bounded symmetric domain?
\end{ques}

In general one can hope that
(\ref{bdd.dom.quot.then.sum}), together with the conjectures
\cite[18.6--8]{shaf-book}, imply that if
 $X$ is a smooth projective variety  whose
 universal cover  is  biholomorphic to a  
semialgebraic subset of a projective variety
then $\pi_1(X)$ is commensurate with an extension of 
an abelian group with a 
cocompact lattice acting on a  bounded symmetric domain.

The second part of the proof in \cite{chk} deals with the case
when  $\pi_1(X)$ is a free  abelian group $\z^{2r}$. Then
the Albanese map
$\alb_X:X\to \Alb(X)\cong \c^r/\z^{2r}$ is the natural candidate
 for the fiber bundle structure
in (\ref{main.thm.quot-version}.2).
In our case, we  construct the non-abelian
Albanese map by
referring to  the works of \cite{eel-sam}  and \cite{siu}.
We thank D.~Toledo for explaining to us that the case when
 $\bd$ is a reducible,
 bounded, symmetric domain of dimension $\geq 2$ and
$\Gamma$ is a cocompact, irreducible lattice acting on
 $\bd$ can be handled using the results of \cite{mok-harmonic}.

Thus we obtain an algebraic morphism  $g:X\to \bd/\Gamma$.
The relative version of the Albanese morphism
should deal with the (almost abelian) kernel of
$\pi_1(X)\to \Gamma$.

The third step is  to prove that the  morphism  $g:X\to \bd/\Gamma$
is a fiber bundle. The proof in \cite{chk} relied on a detailed
knowledge of subvarieties of Abelian varieties and their
finite ramified covers. This is replaced by the
topological arguments of  Section \ref{top.sect}.

While Theorems \ref{homologymap} and \ref{alg.maps.to.K.pi.1.thm}
 are more general than needed for our purposes,
all the examples  suggest that
even stronger results may be true.

\begin{ques}\label{last.q} Let $Y$ be a smooth  projective variety and 
 $g:Y\to X$ a (not necessarily surjective) 
morphism  to a  (not necessarily smooth)  projective variety.
Assume that  the universal cover $\tilde X$ is contractible and that
the fiber product $\tilde Y:=Y\times_X\tilde X$ is homotopic to a
finite CW complex.

Is then $Y$ a differentiable fiber bundle over $X$?
\end{ques}

The example of Kodaira fibrations (see, for instance, \cite[Sec.V.14]{bpv})
 shows that  $g:Y\to X$
need not be a holomorphic fiber bundle.

The above question leads naturally to several interesting problems
concerning the topology of algebraic maps; these are
 discussed in
\cite{bob-kol}.

\begin{ack} We thank B.~Claudon, J.~Fern{\'a}ndez~de~Bobadilla, 
A.~Gabrielov, A.~H\"oring, 
A.~Huckleberry, S.~Krantz, 
C.~Miebach, A.~N\'emethi, D.~Toledo
and  J.~Wahl for useful comments, references and suggestions.
Partial financial support   to JK was provided by  the NSF under grant number 
DMS-0758275.\end{ack}

\def\cprime{$'$} \def\dbar{\leavevmode\hbox to 0pt{\hskip.2ex \accent"16\hss}d}
\providecommand{\bysame}{\leavevmode\hbox to3em{\hrulefill}\thinspace}
\providecommand{\MR}{\relax\ifhmode\unskip\space\fi MR }
\providecommand{\MRhref}[2]{%
  \href{http://www.ams.org/mathscinet-getitem?mr=#1}{#2}
}
\providecommand{\href}[2]{#2}

\bigskip

\noindent Princeton University, Princeton, NJ 08544-1000, USA

\begin{verbatim}kollar@math.princeton.edu\end{verbatim}

\noindent Current address for JP:

\noindent Stanford University, Stanford, CA 94305, USA

\begin{verbatim}pardon@math.stanford.edu\end{verbatim}

\end{document}